\pdfoutput=1
\documentclass[11pt,a4paper]{article}
\usepackage{amsmath,amssymb,amsfonts,mathtools,amsthm,dsfont}
\usepackage{hyperref}
\usepackage{color}

\usepackage{graphicx}
\usepackage{todonotes}
\usepackage{comment}
\usepackage[T1]{fontenc}
\usepackage{authblk}
\usepackage{cite}
\usepackage{mathtools}
\usepackage{bbm}
\usepackage{indentfirst}
\usepackage{enumerate}
\usepackage{mathtools}

\newtheorem{remark}{Remark}
\newtheorem{theorem}{Theorem}
\newtheorem{lemma}{Lemma}
\newtheorem{proposition}{Proposition}

\newcommand{\1}{\ensuremath{\mathds{1}}}
\newcommand{\MM}{\ensuremath{\mathcal{M}}}
\newcommand{\BB}{\ensuremath{\mathcal{B}}}
\newcommand{\BR}{\ensuremath{\mathtt{Bor}}}
\newcommand{\PP}{\ensuremath{\mathcal{P}}}
\newcommand{\CC}{\ensuremath{\mathcal{C}}}
\newcommand{\LL}{\ensuremath{\mathcal{L}}}
\newcommand{\RR}{\ensuremath{\mathbb{R}}}
\newcommand{\NN}{\ensuremath{\mathbb{N}}}

\newcommand{\bk}[2]{\ensuremath{\left\langle #1,#2 \right\rangle}}

\begin{document}
	\title{The e-property of asymptotically stable Markov-Feller operators}

	\author[1]{Ryszard Kukulski\footnote{ryszard.kukulski@gmail.com}}
	\author[1,2]{Hanna Wojew{\'o}dka-{\'S}ci\k{a}{\.z}ko}
	\affil[1]{Institute of Theoretical and Applied Informatics, Polish Academy
		of Sciences,  Ba{\l}tycka 5, 44-100 Gliwice, Poland}
	\affil[2]{Institute of Mathematics, University of Silesia in Katowice, Bankowa 14, 
		40-007 Katowice, Poland}
	\date{}
	\maketitle
	
	\begin{abstract}
		In this work, we prove that any asymptotically stable Markov-Feller 
		operator possesses the e-property everywhere outside at most a meagre 
		set. We also provide an example showing that this result is tight. 
		Moreover, an equivalent criterion for the e-property is proposed.
	\end{abstract}	

\section*{Introduction}
Asymptotic behaviour of random Markov dynamical systems, including mainly the 
existence of stationary distributions, along with asymptotic stability of 
Markov operators acting on measures, associated with these systems, has been 
widely studied over the years. One of the first results concerning asymptotics 
of Markov-Feller operators evolving on Polish metric spaces have been obtained 
by T. Szarek (cf. \cite{szarek_03} or \cite{szarek_06}). Later on, even more 
interesting articles on this topic have been published (see e.g. 
\cite{lasota_szarek_06, szarek_worm_11, czapla_horbacz_14, 
wedrychowicz_18}, just to name a~few). In most of them, the so-called 
lower-bound technique for equicontinuous families of Markov-Feller operators 
has been applied to prove asymptotic stability of these operators. We say that 
a regular Markov operator $P$, with dual operator $U$, has the e-property in 
the set of functions $\mathcal{R}$ if 
the family of iterates $(U^nf)_{n\in\mathbb{N}_0}$ is equicontinuous for all 
$f\in\mathcal{R}$. Most often, $\mathcal{R}$ is assumed to be the set of all 
bounded Lipschitz functions, as e.g. in \cite{worm_10, szarek_worm_11, 
hille_szarek_ziem_17}, although it can also viewed as the set of all bounded 
continuous functions, as in this paper (for the convenience of the reader, both 
these cases are discussed and compared in Remark \ref{rem:1}).  
Markov operators with the e-property are widely 
applied e.g. in the theory of partial differential equations (cf.  
\cite{lasota_szarek_06, komorowski_peszat_szarek_10, szarek_14} and the papers 
relating to the equation of a passive tracer \cite{szarek_sleczka_urbanski_10} 
or a non-linear heat equation driven by an impulsive noise  
\cite{kapica_szarek_sleczka_11}). On the other hand, similar techniques as 
those described above have been also applied to establish the existence of a 
unique stationary distribution in a stochastic model for an autoregulated gene 
\cite{hille_horbacz_szarek_16}. Incidentally, let us draw the attention of the reader to the fact that asymptotic stability, or even exponential ergodicity, for a~general class of Markov operators may be also proved using a quite different concept, 
based on the application of an asymptotic coupling (cf. 
\cite{czapla_18, hairer_02,hairer_cloez_15, kapica_sleczka_20,czapla_horbacz_w-s_19}).

Knowing the criteria on asymptotic stability of Markov-Feller operators possessing the e-property, one may ask about a reverse relation, studied e.g. in \cite{hille_szarek_ziem_17}, where the authors prove that any asymptotically stable Markov-Feller operator with an invariant measure such that the interior of its support is non-empty satisfies the e-property. In this paper, we generalize  this result (formulated as \hbox{\cite[Theorem 2.3]{hille_szarek_ziem_17}}). To be more precise, we prove that any asymptotically stable Markov-Feller operator possesses the e-property everywhere except at most a meagre set (Theorem \ref{th-2}). 
Moreover, in Theorem \ref{th-1}, we propose an equivalent condition for the 
e-property for asymptotically stable Markov-Feller operators. Namely, we prove 
that any asymptotically stable Markov-Feller operator has the e-property if and 
only if it has the e-property in at least one point of the support of its 
invariant measure. 
Our mains results, that is, Theorems \ref{th-2} and \ref{th-1}, then naturally 
imply \cite[Theorem 2.3]{hille_szarek_ziem_17}. Indeed, it is clear, according 
to Theorem \ref{th-2}, that, whenever the interior of the support of an 
invariant measure of a Markov-Feller operator $P$ is non-empty, then there exists at 
least one point belonging to this support, at which $P$ has the e-property. 
This, in turn, implies, due to Theorem \ref{th-1}, that $P$ possesses the 
e-property at any point.

In the final part of the paper, we present two examples. In 
\hyperref[ex1]{Example 1}, we define an asymptotically stable Markov-Feller 
operator such that the set of points, at which it does not possess the 
e-property, is a dense set. Such an example yields that the main result of this 
paper, formulated as Theorem \ref{th-2}, is tight. In \hyperref[ex2]{Example 
2}, on the other hand, we construct an asymptotically stable Markov-Feller 
operator, for which the set of points not possessing the e-property is 
uncountable. 

The outline of the paper is as follows. Section \ref{sec:1} contains notation and basic definitions relating mainly to the theory of Markov operators. In Section \ref{sec:main_results}, we state the main results of this article, as well as conduct their proofs. Section \ref{sec:examples} is devoted to the above-mentioned examples, which complete the discussion on the e-property of asymptotically stable Markov-Feller operators.

\section{Preliminaries}\label{sec:1}
	Let $(S,\rho)$ be a Polish metric space. By $B(x,\epsilon)$ we denote an open ball in $S$ centered at $x\in S$ and of radius $\epsilon>0$. Closure and interior  of any set $A \subset S$ shall be denoted by $\mathtt{Cl}(A)$ and $\mathtt{Int} (A)$, respectively.

	Let $(S,\rho)$ be endowed with a Borel $\sigma-$algebra $\BR(S)$. Now, let $\BB_b(S)$ be a family of all real-valued, bounded and Borel measurable functions on $S$, equipped with the supremum norm $\|f\|:=\sup_{x\in S}|f(x)|$, and let $\CC_b(S)$ and $\LL_b(S)$ be the subfamilies of $\BB_b(S)$ consisting of continuous and Lipschitz continuous functions, respectively. By $\LL_{FM}(S)$ we mean the special subfamily of $\LL_b(S)$ whose components satisfy $f(S) \subset [0,1]$ and $\mathtt{Lip}(f)\le 1$, where $\mathtt{Lip}(f)$ denotes the Lipschitz constant of $f$.
	
	Let us further consider the set $\MM(S)$ of finite Borel measures, defined on the measurable space $(S,\BR(S))$, and its subset $\MM_1(S)$ consisting of probability  measures. Moreover, we will also consider the linear space $\MM_s(S)$ of finite signed Borel measures, i.e.
	\begin{equation*}
	\MM_s(S)=\{\mu=\mu_+ - \mu_-: \; \mu_+,\mu_- \in \MM(S) \}.
	\end{equation*}
	Let us equip this space with the Fortet-Mourier norm 
$\| \cdot \|_{FM}$, defined by
	\begin{equation*}
	\|\mu\|_{FM}=\sup_{f \in \LL_{FM}(S)} |\bk{f}{\mu}|, \quad \mu \in \MM_s(S),
	\end{equation*}
	where $\bk{f}{\mu} := \int_S f(x) \mu(dx)$ for any $f 
	\in \BB_b(S)$, $\mu \in \MM_s(S)$. The support of any measure $\mu \in \MM(S)$ shall be defined as usual, that is
	\begin{equation*}
	\text{supp} \; \mu=\{x \in S: \mu(B(x,\epsilon))>0 \mbox{ for any } \epsilon>0\}.
	\end{equation*}
	The operator $\PP: \MM(S) \rightarrow \MM(S)$ is called Markov if
		\begin{itemize}
			\item $\PP( \lambda \mu_1 + \mu_2)= \lambda \PP(\mu_1)+ \PP(\mu_2)$ 
			for any $\lambda \ge 0$ and any $\mu_1, \mu_2 \in \MM(S)$,
			\item $\PP \mu (S) = \mu(S)$  for any $\mu \in \MM(S)$.
		\end{itemize}
	We say that a Markov operator $P$ is regular, provided that there exists a~linear map (the 
	so-called dual operator) 
	$U: \BB_b(S) \rightarrow \BB_b(S)$ such that
		\begin{equation*}
		\bk{f}{\PP\mu}=\bk{Uf}{\mu} \quad \text{for any } f\in 
		\BB_b(S),\mu\in\MM(S).
		\end{equation*}
	 In this work, we will focus on Markov-Feller operators, that is, regular Markov operators that fulfill the property $U(\CC_b(S)) \subset \CC_b(S)$. 
	 
	 Let us also indicate that, given a transition probability function 
	 \linebreak\hbox{$\pi:S\times \BR(S)\to[0,1]$}, that is, a map for which 
	 $\pi(x,\cdot):\BR(S)\to[0,1]$ is a probability measure for any fixed $x\in 
	 S$ and $\pi(\cdot,A):S\to[0,1]$ is a~Borel measurable function for any 
	 fixed $A\in\BR(S)$, one may define a~regular Markov operator $P$, along 
	 with its dual operator $U$, as follows:
	 \begin{align}\label{pi_P_U}
	 \begin{aligned}
	 &P\mu(A)=\left\langle\pi(\cdot,A),\mu\right\rangle
	 \quad\text{for any }A\in\BR(S),\;\mu\in\MM(S)\\
	 &Uf(x)=\left\langle f,\pi(x,\cdot)\right\rangle
	 \quad\text{for any }x\in S,\;f\in\BB_b(S).
	 \end{aligned}
	 \end{align}

As we have already mentioned before, we will study the relation between two properties of Markov-Feller operators: asymptotic stability and the \hbox{e-property}.  
	 We say that a sequence $(\mu_n)_{n \in \NN}$ of finite Borel measures on~$S$ converges weakly to a measure $\mu \in \MM(S)$ (which we denote by $\mu_n \xrightarrow{\omega} \mu$), as $n\to\infty$, if for any $f\in \CC_b(S)$ we 
	 have
	 \begin{equation*}
	 \lim_{n \to \infty}\bk{f}{\mu_n}= \bk{f}{\mu}.
	 \end{equation*}
	A Markov operator $\PP$ is said to be asymptotically stable if there exists a~unique measure \hbox{$\mu_* \in \MM_1(S)$} such that $\PP \mu_* = \mu_*$ ($\mu_*$ is then called an invariant measure of $P$) and $\PP^n\mu \xrightarrow{\omega} \mu_*$, as $n \to \infty$, for each measure $\mu \in \MM_1(S)$. 	 
	 A Markov operator $\PP$ has the e-property in a set of functions 
	 $\mathcal{R}$ at a point $z \in S$, if for any $f \in \mathcal{R}$ the 
	 following holds:
	 \begin{equation}\label{e-prop}
	 \lim_{x \rightarrow z}\sup_{n \in \NN} |U^nf(x)-U^nf(z)|=0.
	 \end{equation}
	 If the above equality holds for each $z \in S$, then we say that $\PP$ has 
	 the \hbox{e-property} in $\mathcal{R}$. Usually, $\mathcal{R}$ is assumed 
	 to be one of the following family of functions: $\CC_b(S)$, $\LL_b(S)$, 
	 $\LL_{FM}(S)$.
	 \begin{remark}\label{rem:1}
	 	In this paper, the notion of the e-property in $\LL_{FM}(S)$ shall be 
	 	needed in the proof of Theorem \ref{th-2}. Let us, however, observe 
	 	that whenever \eqref{e-prop} is satisfied for every $f\in \LL_{FM}(S)$, 
	 	it is also satisfied for every \hbox{$f\in \LL_b(S)$}, due to the 
	 	linearity of $U$. This implies that the notions of the 
	 	\hbox{e-property} in $\LL_{FM}(S)$ and $\LL_b(S)$ coincide. 
	 	Nevertheless, if no further assumptions are imposed, a regular Markov operator does not need to have the e-property in $\CC_b(S)$ at any point, even if it possesses the e-property in $\LL_{FM}(S)$ at every point. 
Indeed, in the case where $S=\RR$ and a Markov operator $\PP$ is given by $\PP \mu = \mu \circ T^{-1}$ for $\mu \in \MM(S)$, with $T: S \to S$ defined 	by $T(x)=x+1$ for $x \in S$, one can see that $\PP$ has the e-property in $\LL_{FM}(S)$. In spite of that, for each $z \in S$ and a function $f_z \in \CC_b(S)$, given as 
	 	\begin{equation*}
	 	f_z(x)=(n+2)^2 (x-(z+n)) (z+n+2/(n+2) - x)
	 	\end{equation*}
for $x \in \left[z+n,z+n+2/(n+2) \right)$, $n \in \NN$, and $f_z(x)=0$ everywhere else in $S$, we obtain
	 	\begin{equation*}
	 	\begin{split}
	 	&\limsup_{m \to \infty} \sup_{n \in 
	 	\NN}|U^nf_z(z+1/(m+2))-U^nf_z(z)| \\
	 	&\qquad\geq \limsup_{m \to \infty} |f_z(z+m+1/(m+2))-f_z(z+m)|=1,
	 	\end{split}	 	
	 	\end{equation*}
which means that $\PP$ does not have the e-property in $\CC_b(S)$ at any $z \in S$.
	 	
	 	On the other hand, let us indicate that the notion of the e-property in 
	 	$\CC_b(S)$ coincides with the corresponding ones in $\LL_{FM}(S)$ and 
	 	$\LL_b(S)$, provided that a given Markov operator is asymptotically 
	 	stable, which fact shall be proved later on in Lemma \ref{lem-2}.
	 \end{remark}

\section{Main results}\label{sec:main_results}

In this section, we will formulate and prove the main results of this paper.

\begin{theorem}\label{th-2}
	Let $\PP$ be an asymptotically stable Markov-Feller operator. The set of 
	points, where $\PP$ does not have the e-property in $\CC_b(S)$, is a~meagre 
	set, while the set of points, at which $P$  possesses the e-property in 
	$\CC_b(S)$, is dense.
\end{theorem}

Before we prove Theorem \ref{th-2}, let us first establish a few lemmas, to which we will refer in the main proof.

\begin{proposition}\label{prop-1}
	Let $\PP$ be a regular Markov operator, which has the e-property in 
	$\LL_{FM}(S)$ at $z\in S$. Then any map $S\ni x\mapsto\PP^n 
	\delta_x\in\mathcal{M}_1(S)$, $n \in \NN$, is a~continuous function in the 
	space $\MM_s(S)$, equipped with the weak topology, at $z \in S$.
\end{proposition}
\begin{proof}
	Fix $n \in \NN$, and consider a sequence $(x_m)_{m \in \NN}$ of points 
	converging to $z \in S$. Due to the e-property of $\PP$ in $\LL_{FM}(S)$ at 
	$z\in S$, for each $g \in \LL_{FM}(S)$, we have
	\begin{equation*}
	\lim_{m \to 
		\infty} \sup_{n \in \NN} |U^ng(x_m)-U^ng(z)|=0,
	\end{equation*}
	which implies that $\bk{g}{\PP^n \delta_{x_m}} \to \bk{g}{\PP^n \delta_{z}}$, as $m \to \infty$. Finally, due 
to the Portmanteau Theorem (\cite[Theorem 13.16]{klenke_13}), we receive $\PP^n \delta_{x_m} 	\xrightarrow{\omega} \PP^n \delta_z$, as $m \to \infty$.
\end{proof}

\begin{lemma}\label{lem-1}
	If $f \in \CC_b(S)$ and $K \subset S$ is an arbitrary compact set, then, for each $\epsilon>0$, there exists $L \in \LL_b(S)$ such that $\| f|_K - L|_K\| \leq \epsilon$ and $\|L\| \leq \|f\|$.
\end{lemma}
\begin{proof}
	Choose an arbitrary $f\in \CC_b(S)$, and fix $\epsilon>0$. Function $f|_K$ 
	is uniformly continuous. Hence, for $\delta<\epsilon$, there exists $r>0$ 
	such that \hbox{$|f(x)-f(y)| \leq \delta$} for any $x,y \in K$, satisfying 
	$\rho(x,y)\leq r$. Due to the compactness 
of $K$, it is possible to find a finite cover of $K$, i.e.
	\begin{equation*}
		K=\bigcup_{i=1}^N \left(\mathtt{Cl}(B(x_i,r/2)) \cap K \right) \quad 
		\text{for some } N \in \NN,
	\end{equation*}
	where $x_1,\ldots,x_N \in K$.
	
	Let us define a family of real functions 
$\{L^{c,l}:c,l>0 \}$ given by the formula
	\begin{equation*}
		L^{c,l}(x)=\sum_{i=1}^N p_i^{c,l}(x) f(x_i)\quad\text{for }x\in S,
	\end{equation*}
	where
	\begin{equation*}
 	\begin{split}
p_i^{c,l}(x)&=\frac{d_i^l(x) + c/N}{\sum_{j=1}^N d_j^l(x) + c},\\
d_i^l(x)&=\frac{1}{\max\{l (\rho(x,x_i)-r/2),0\}+1}.
	\end{split}
	\end{equation*}
	Note that, for any $c,l>0$, $(p_i^{c,l}(x))_{i=1}^N$ is a probability vector, whence $\|L^{c,l}\|\leq \|f\|$. The function $x \mapsto \max\{l(\rho(x,x_i)-r/2),0\}+1 \geq 1$ is 
	Lipschitz continuous, and so is a~function $x \mapsto d_i^l(x) \in [0,1]$. Therefore, we see that $x \mapsto 1/(\sum_{j=1}^N d_j^l(x)+c)$ is Lipschitz continuous. 
	Moreover, it is bounded. Finally, we observe that $p_i^{c,l}$ is Lipschitz continuous, too, and so is a function $L^{c,l}$.

	Take an arbitrary $x \in K$. Without loss of generality, we can assume that $\rho(x,x_1) \leq r/2$,  and there exists $J \in \{1,\ldots,N\}$ such that $\rho(x,x_i) \leq r$ for any $i=1,\ldots,J$, and $\rho(x,x_i)>r$ for any $i>J$. We therefore obtain
	\begin{equation*}
		\begin{split}
			|L^{c,l}(x)-f(x)| &\leq \sum_{i=1}^N p_i^{c,l}(x) |f(x_i)-f(x)| 
\leq 
			\sum_{i=1}^J 
			p_i^{c,l}(x) \delta + \sum_{i=J+1}^N p_i^{c,l}(x) 2 \|f\| \\
			&\leq \delta + 2 \|f\| 
			\sum_{i=J+1}^N \frac{d_i^l(x) + c/N}{\sum_{j=1}^N d_j^l(x) + c}.
		\end{split}
	\end{equation*}
	Note that $d_1^l(x)=1$, and also, for $i > J$, we have
	\begin{equation*}
		d_i^l(x) \leq \frac{1}{lr/2+1}.
	\end{equation*}
	Finally, we get
	\begin{equation*}
		\begin{split}
			|L^{c,l}(x)-f(x)| &\leq \delta + 2 \|f\| 
			\sum_{i=J+1}^N \frac{d_i^l(x) + c/N}{\sum_{j=1}^N d_j^l(x) + c} 
			\\ & \leq \delta + 2 
			\|f\| (N-J) \left(\frac{1}{lr/2+1} + \frac{c}{N}\right)\\
			&\leq \delta + 2 
			\|f\| (N-1) \left(\frac{1}{lr/2+1} + \frac{c}{N}\right),
		\end{split}
	\end{equation*}
	and
	\begin{equation*}
		\lim_{c \to 0, l \to \infty} \delta + 2 
		\|f\| (N-1) \left(\frac{1}{lr/2+1} + \frac{c}{N}\right) = \delta < 
\epsilon,
	\end{equation*}
	which means that it is possible to choose $c_0,l_0>0$ so that
	\begin{equation*}
		\delta + 2 
		\|f\| (N-1) \left(\frac{1}{l_0r/2+1} + \frac{c_0}{N}\right)<\epsilon.
	\end{equation*}
	Consequently, the function $L \coloneqq L^{c_0,l_0}$ satisfies 
	$\|L|_K-f|_K\| 
	\leq 
	\epsilon,$ and the proof is completed.
\end{proof}
\begin{lemma}\label{lem-2}
	A regular Markov operator $\PP$, which is asymptotically stable and has the 
	e-property in $\LL_{FM}(S)$ at $z \in S$, possesses also the e-property in 
	$\CC_b(S)$ at $z \in S$.
\end{lemma}
\begin{proof}
	By assumption, for $z\in S$ and any $g\in \LL_{FM}(S)$, we have
	\begin{equation*}
	\lim_{x \to z} \sup_{n \in \NN} |U^ng(x)-U^ng(z)|=0.
	\end{equation*}
	The above equality also holds for any $g\in \LL_b(S)$, as it was already 
	explained in Remark \ref{rem:1}.
	
	Choose an arbitrary function $f \in \CC_b(S)$, and let $(x_m)_{m\in \NN}$ be a sequence of points from $S$, converging to $z $, as $m \to \infty$. Let us further consider the family of probability measures 
	\begin{equation*}
		\mathcal{M}= \{\mu_*\} \cup \{ \PP^n \delta_z: n \in \NN \} \cup 
		\{\PP^n 
		\delta_{x_m}: n,m \in \NN \},
	\end{equation*}
	where $\mu_*$ stands for the unique invariant measure of $\PP$. We want to prove that $\MM$ is compact in the space $\MM_s(S)$, equipped with the weak topology. Choose an arbitrary sequence of measures $(\mu_k)_{k \in \NN} \subset \MM$. There are three non-trivial cases to consider:
	\begin{itemize}
		\item The sequence $(\mu_k)_{k \in \NN}$ contains infinitely many elements either of the sequence $(\PP^n 
\delta_z)_{n \in \NN}$ or $(\PP^n \delta_{x_m})_{n \in \NN}$ for some fixed $m \in \NN$. Due to the asymptotic stability of $\PP$, both these sequences converge weakly to $\mu_*$, as $n \to \infty$.
		\item  The sequence $(\mu_k)_{k \in \NN}$ contains infinitely many elements of the sequence $(\PP^n\delta_{x_m})_{m \in \NN}$ for some fixed $n \in \NN$. By Proposition \ref{prop-1}, we have $\PP^n \delta_{x_m} \xrightarrow{\omega} \PP^n \delta_{z},$ as $m \to \infty$.
		\item The sequence $(\mu_k)_{k \in \NN}$ contains infinitely many elements of the sequence $(\PP^{n_m} \delta_{x_m})_{m \in \NN}$, where $n_m \nearrow \infty$, as $m \to \infty$. Then, for any $g \in\LL_{FM}(S)$, we have
		\begin{equation*}
		\begin{split}
			\lim_{m \to \infty} |\bk{g}{\PP^{n_m} \delta_{x_m}}-\bk{g}{\mu_*}| 
		&\leq 
		\lim_{m \to \infty} |\bk{g}{\PP^{n_m} 
			\delta_{x_m}}-\bk{g}{\PP^{n_m} 
			\delta_z}| \\
		&+\lim_{m \to \infty} |\bk{g}{\PP^{n_m} 
			\delta_z}-\bk{g}{\mu_*}| =0,
		\end{split}
		\end{equation*}
		where the last equality follows from the asymptotic stability of $\PP$. Due 
		to the Portmanteau Theorem, we further obtain $\PP^{n_m} \delta_{x_m} 
		\xrightarrow{\omega} \mu_*,$ as $m \to \infty$.
	\end{itemize}
	Using the Prokhorov theorem (\cite[Theorem 5.1 \& 5.2]{billingsley_99}), we obtain that the family $\MM$ is tight. 
Hence, for an arbitrairly fixed $\epsilon>0$, there exists a~compact set $K \subset S$ such that for any $\mu \in \MM$ we have $\mu(K') \leq \epsilon$.

	According to Lemma \ref{lem-1}, for a given function $f\in\CC_b(S)$, there exists a~function $L \in \LL_b(S)$, which satisfies $\|f|_K - L|_K \|\leq \epsilon$ and $\|L\| \leq \|f\|$, whence we get
	\begin{equation*}
	L|_K+f|_{K'}-\epsilon \leq f\leq L|_K+f|_{K'}+\epsilon.
	\end{equation*}
	This, in turn, implies
	\newlength{\twidth}
	\settowidth{\twidth}{$\le |$}
	\begin{equation*}
		\begin{split}
			&|\bk{f}{\PP^n \delta_{x_m}-\PP^n 
			\delta_z}| \\
			&\qquad\leq |\bk{L|_K}{\PP^n 
				\delta_{x_m}-\PP^n \delta_z}|+ |\bk{f|_{K'}}{\PP^n 
\delta_{x_m}}| 
			+|\bk{f|_{K'}}{\PP^n \delta_z}| +2 \epsilon \\
			&\qquad\leq |\bk{L}{\PP^n 
				\delta_{x_m}-\PP^n \delta_z}|+ |\bk{L|_{K'}}{\PP^n 
				\delta_{x_m}-\PP^n \delta_z}|+ 2\epsilon(1+\|f\|)\\
			&\qquad\leq |\bk{L}{\PP^n 
				\delta_{x_m}}-\bk{L}{\PP^n \delta_z}|+ 2\epsilon(1+2\|f\|).
		\end{split}
	\end{equation*}
	Finally, we obtain
	\begin{equation*}
		\begin{split}
		&\limsup_{ m \to \infty } \sup_{n \in \NN} 
		|U^nf(x_m)-U^nf(z)| \\
			&\qquad\leq 
\limsup_{ m 
				\to 
				\infty } \sup_{n \in \NN} |U^nL(x_m)-U^nL(z)|+ 
2\epsilon(1+2\|f\|)\\
			&\qquad= 2\epsilon(1+2\|f\|),
		\end{split}
	\end{equation*}
which, in view of the fact that $\epsilon$ was chosen arbitrarily, proves that 
$\PP$ has the e-property in $\CC_b(S)$ at $z\in S$.
\end{proof}

Now, we are ready to prove the assertion of Theorem \ref{th-2}.
\begin{proof}[Proof of Theorem \ref{th-2}]
	Let $\mu_* \in \MM_1(S)$ be a unique invariant measure of $\PP$. By $\widehat{O}_x$ we will denote an open neighborhood of $x \in S$, containing $x$. For each $k \in \NN$ let us define an open set
	\begin{equation*}
		O_k=\left\{x \in S: \, \exists_{\widehat{O}_x} \, \exists_{n_x 
\in \NN} \, \forall_{n \geq n_x}  \, \forall_{y \in \widehat{O}_x} \quad 
\|\PP^n 
		\delta_y - \mu_*\|_{FM} \leq 1/k \right\}.
	\end{equation*}
	We want to show, that, for any $k\in \NN$, the set $O_k$ is dense in $S$. Let us take $x_0\in S$ and $\epsilon>0$. Define
	\begin{equation*}
		\begin{split}
			&Y\coloneqq \overline{B(x_0,\epsilon)},\\
			&Y_{k,n} \coloneqq \left\{x \in Y: \forall_{m \geq n} \; \|\PP^m \delta_x - \mu_* \|_{FM} 
			\leq 1/k \right\}, \quad  k,n \in \NN.
		\end{split}
	\end{equation*}
	One can note that $Y$, as a closed subset of a Polish space, is a Polish space, and the sets $Y_{k,n}$, $k,n \in \NN$, are closed by the Feller property of $P$. By assumption, $\PP$ is asymptotically stable, so $\PP^n 
\delta_x \xrightarrow{\omega} \mu_*$, as $n \to \infty$, which is equivalent to $\|\PP^n \delta_x - \mu_*\|_{FM} \to 0$ (cf. \cite[Theorem 8.3.2.]{bogachev_07}). That means that $Y= \bigcup_{n \in \NN} Y_{k,n}$ for any $k\in\NN$. By the Baire category theorem, Polish spaces are necessarily Baire spaces, so, for each $k\in\NN$, there exists $N \in \NN$ such that $\mathtt{Int} (Y_{k,N}) \not= \emptyset$. Obviously, 
$\mathtt{Int} (Y_{k,N}) \subset Y$ and $\mathtt{Int} (Y_{k,N}) \subset O_k$, so the intersection of $O_k$ with an arbitrarily small open ball $Y$ is non-empty. This, in turn, implies that, for any $k\in\NN$, the set $O_k$ is dense in $S$. 

Let us now define 
	\begin{equation*}
		C \coloneqq \bigcap_{k \in \NN} O_k.
	\end{equation*}
Fix $z\in C$. Obviously, $z\in O_k$ for any $k\in\NN$. Further, let $\epsilon>0$ and $k \in \NN$ be such that $1/k \leq \epsilon/2$, and note that, by the definition of $O_k$, there exist an open neighbourhood $\widehat{O}_z$ of $z$, as well as $n_z \in\NN$ such that for any $n\geq n_z$ and any $x\in\widehat{O}_z$ one has 
	\begin{equation*}
		\|\PP^n \delta_x - \mu_*\|_{FM} \leq 1/k \leq \epsilon/2.
	\end{equation*}
	Due to the asymptotic stability of $\PP$, there exists $n_0 \in \NN$ such that \linebreak\hbox{$\|\PP^n \delta_z - \mu_* \|_{FM} \leq \epsilon/2$} for any $n \geq n_0$. 
	Using the Feller property of $\PP$, we can conclude that there exists another neighborhood $\widetilde{O}_z$ of point $z$ such that for every $n < \max(n_0,n_z)$ and every $x \in \widetilde{O}_z$ we have 
	$\|\PP^n \delta_x 
	- 
	\PP^n \delta_z \|_{FM} \leq \epsilon.$ 
	Hence, for $x \in \widehat{O}_z \cap 
	\widetilde{O}_z $ we obtain
	\settowidth{\twidth}{$\le |$}
	\begin{equation*}
	\begin{split}
	&  \sup_{n \in \NN} \|\PP^n \delta_x - \PP^n \delta_z 
	\|_{FM} \\
	&\qquad\leq 
	\max\left(\epsilon, \sup_{n \geq \max(n_0,n_z)} \|\PP^n \delta_x - \mu_* 
	\|_{FM}+\sup_{n \geq \max(n_0,n_z)} \|\PP^n \delta_z - \mu_* \|_{FM}\right) 
	\\ &\qquad\leq \epsilon.
	\end{split}
	\end{equation*}
	Keeping in mind that $\epsilon>0$ and $z\in C$ were chosen arbitrarily, we 
	end up with the following equality:
	\begin{equation*}
		\lim_{x \to z} \sup_{n \in \NN} \| \PP^n \delta_x - \PP^n \delta_z 
\|_{FM}=0\quad\text{for }z\in C,
	\end{equation*}
	which means that $\PP$ has the e-property in $\LL_{FM}(S)$ at any $z \in 
	C$. Further, 
Lemma \ref{lem-2} yields that $P$ also enjoys the e-property in $\CC_b(S)$ at 
each $z\in C$.  
	As a consequence, the set of points, where $\PP$ does not have the 
	e-property in $\CC_b(S)$, constituting a subset of 
	$C^{\prime}=\bigcup_{k\in\NN}O_k^{\prime}$, where 
	$O_k^{\prime}$, $k\in\NN$, are nowhere dense, is a meagre set. Moreover, 
	using again the Baire category theorem, we obtain that the set $C$ is 
	dense. 
\end{proof}

\begin{theorem}\label{th-1}
	Let $\PP$ be an asymptotically stable Markov-Feller operator with an 
	invariant measure $\mu_* \in \MM_1(S)$. Then $\PP$ has the e-property in 
	$\CC_b(S)$ if there exists at least one point $z \in \text{supp}\, \mu_*$, 
	at which $\PP$ has the e-property in $\CC_b(S)$.
\end{theorem}

Before we present the sketch of the proof of Theorem \ref{th-1}, let us quote 
the main result of \cite{hille_szarek_ziem_17}.

\begin{theorem}\cite[Theorem 
	2.3]{hille_szarek_ziem_17}\label{th-szarek}
	Let $\PP$ be an asymptotically stable Markov-Feller 
	operator and let $\mu_*$ be its invariant measure. If 
\hbox{$\mathtt{Int}(\mathtt{supp} \, \mu_*) \not= 		\emptyset$}, then $\PP$ 
satisfies the e-property in $\CC_b(S)$.
\end{theorem}

The proof of Theorem \ref{th-szarek} is based on the application of the following lemma:

	\begin{lemma}\cite[Lemma 2.4]{hille_szarek_ziem_17}\label{lem-szarek}
	Let $\PP$ be an asymptotically stable Markov-Feller operator whose unique invariant measure is denoted by $\mu_* \in \MM_1(S)$. If $\mathtt{Int} (\mathtt{supp} \, \mu_*) \not= \emptyset$, then for an arbitrary function $f \in \CC_b(S)$ 
	and any $\epsilon>0$, there exist a ball $B \subset \mathtt{supp} \, \mu_*$ 
	and a constant $N \in \NN$ such that for any $n \geq N$ and any $x \in B$, we have 
	$|U^nf(x)-\bk{f}{\mu_*}| \leq \epsilon$.		
\end{lemma}

\begin{proof}[The sketch of the proof of Theorem \ref{th-1}]
	Following the proof of Lemma \ref{lem-szarek}, one may observe that its 
	assertion holds (excluding condition 
	$B \subset \mathtt{supp} \, \mu_*$) without 
	assuming $\mathtt{Int} (\mathtt{supp} \, \mu_*) 		\not= \emptyset$. 
	On the other hand, 
	after analyzing the proof of Theorem~\ref{th-szarek}, 		we come to the 
	conclusion that the e-property of $P$ in $\CC_b(S)$ can be relatively 
	easily obtained under the following condition:
	\begin{equation}\label{eq-th-1-2}
	\forall_{f \in \CC_b(S)} \, \forall_{\epsilon>0} \, \exists_{B 
		\subset S: \mu_*(B)>0}
	\, \exists_{N \in \NN} \, \forall_{n \ge N} \, \forall_{x \in B} \quad 
	|U^nf(x)-\bk{f}{\mu_*}| \le \epsilon.
	\end{equation}
	Hence, proving that $\mu_*(B)>0$ is still crucial, and it shall to be done here without assuming that $\mathtt{Int} (\mathtt{supp} \, \mu_*) \not= \emptyset$.

	In view of the above, it suffices now to show that any asymptotically 
	stable Markov-Feller operator $\PP$, possessing the e-property in 
	$\CC_b(S)$ at some point \hbox{$z \in \mathtt{supp} \, \mu_*$}, satisfies 
	\eqref{eq-th-1-2}. 
	
	Let $f \in	
	\CC_b(S)$ and $\epsilon>0$. Then, by assumption, there exist $z \in 
	\text{supp } \, \mu_*$ such that	
	\begin{equation*}
	\lim_{x \rightarrow z} \sup_{n \in \NN} |U^nf(x)-U^nf(z)|=0.
	\end{equation*}
	As a consequence, setting $B \coloneqq \widehat{O}_z$, where $\widehat{O}_z$ is an open neighbourhood of $z$ such that for any $x \in \widehat{O}_z$ we have
	\begin{equation*}
	\sup_{n \in \NN} |U^nf(x)-U^nf(z)| \leq \frac{\epsilon}{2},
	\end{equation*}
	we get $\mu_*(B)>0$. By the asymptotic stability of $\PP$, we can take $N \in \NN$ such that for $n \geq N$ the following holds:
	\begin{equation*}
	|U^nf(z)-\bk{f}{\mu_*}|\le \epsilon/2.
	\end{equation*}
	Finally, for any $x \in B$ and any $n \geq N$, we obtain 
	\begin{equation*}
	|U^nf(x)-\bk{f}{\mu_*}| \le |U^nf(x)-U^nf(z)| + |U^nf(z)-\bk{f}{\mu_*}| 
	\le 
	\epsilon,
	\end{equation*}
	which gives \eqref{eq-th-1-2}, and hence completes the sketch of the proof.
\end{proof}

\section{Examples}\label{sec:examples}
	
	In this section we shall present two important examples. 
	
	In \hyperref[ex1]{Example 1} we construct an asymptotically stable 
	Markov-Feller operator $P$ such that the set of points, at which $P$ does 
	not possess the e-property in $\CC_b(S)$, is dense, and hence it is a 
	non-trivial meagre set. The main aim of presenting such an example is to 
	justify that the main result of this paper, formulated as Theorem 
	\ref{th-2}, is tight.
	
	In \hyperref[ex2]{Example 2}	
	we construct an asymptotically stable Markov-Feller operator $\PP$ such 
	that the set of points, where $\PP$ does not possess the e-property in 
	$\CC_b(S)$, has 
positive Lebesgue measure, and so, it is, in particular, uncountable.

\subsection*{Example 1}\label{ex1}

	Let $S$ be a unit sphere in $\RR^2$ endowed with the Euclidean metric, i.e.
		\begin{equation*}
		S=\{\phi(x)=(cos(2 \pi x), \sin(2 \pi x)): x \in [0,1)\}.
		\end{equation*}
		Then $S$ is obviously a Polish space. Note that every number $x \in 
		[0,1)$ can be uniquely represented, using binary numeral system, as
		\begin{equation*}
		\sum_{i=1}^{\infty} \frac{1}{2^i} e_i (x) \eqqcolon 
		[e_1(x),e_2(x),\ldots]_2, \quad e_i(x) \in \{0,1\},
		\end{equation*}
		according to the following convention:
		\begin{equation*}
		\forall_{x \in [0,1)} \quad \exists_{(i_k)_{k \in \NN} 
			\subset \NN:\; i_k \nearrow \infty} \quad \forall_{k \in \NN} \quad 
			e_{i_k}(x)=0.
		\end{equation*}
		Since $\phi:[0,1)\to S$ is a bijection, from now on, we will identify 
		any $\phi(x)\in S$ with $x\in[0,1)$.
		
		Let us further define $\pi: [0,1) \times \BR([0,1)) \to \RR$ by the 
		following 
		formula:
		\begin{equation*}
		\pi(y,\cdot)=\begin{cases} \delta_{2y}(\cdot) & \quad \text{for } 
		0\leq y<1/2\\
		\frac{1}{2} \delta_0(\cdot) + \frac{1}{2} 
		\delta_{2y-1}(\cdot) & \quad \text{for } 1/2\leq y < 1.
		\end{cases}
		\end{equation*}
		Equivalently, (in the binary notation) we can write
		\begin{equation*}
		\begin{split}
		\pi\left([0,e_1(x),e_2(x),\ldots]_2,\cdot\right)&=\delta_{[e_1(x),e_2(x),\ldots]_2}(\cdot)\\
		\pi\left([1,e_1(x),e_2(x),\ldots]_2,\cdot\right)&=\frac{1}{2} 
		\delta_0(\cdot) + \frac{1}{2} 
		\delta_{[e_1(x),e_2(x),\ldots]_2}(\cdot) \quad \text{for } x \in S.
		\end{split}
		\end{equation*}
		Note that $\pi$ is a transition probability function, so we can define 
		$\PP$ and $U$ to be a Markov operator and its dual operator, 
		respectively, both  generated by $\pi$, according to the rule given in 
		\eqref{pi_P_U}.
		
		Let $f \in \CC_b(S)$. For any $y \in [0,1/2)$ a map $y\mapsto 
		Uf(y)=f(2y)$ is 
		continuous on the set $(0,1/2)$ and right continuous in $0$. Similarly, for any $y \in 
		[1/2,1)$ a function $y\mapsto Uf(y)=(f(0)+f(2y-1))/2$ is continuous on $(1/2,1)$ 
		and right continuous in $1/2$.
		Moreover, if $y \nearrow 1$, then \hbox{$Uf(y) \to f(0)=Uf(0)$}, and if 
		$y \nearrow 1/2$, then $Uf(y) \to f(0)=Uf(1/2)$. That means that $\PP$ 
		is a Markov-Feller operator. 
		One can also observe that $\delta_0$ is an invariant measure of $\PP$. 
		
		Now, choose arbitrarily $\mu \in \MM_1(S)$, $A \in \BR(S)$ and $\epsilon>0$. 
		Let $N \in \NN$ be such that $2^{-N}<\epsilon,$ and, for any $n \in 
		\NN$, define
		\begin{equation*}
		B_n=\left\{ x \in [0,1):\; \sum_{i=1}^n e_i(x) \geq N \;\vee\; 
		\sum_{i=n+1}^\infty 
		e_i(x)=0 \right\}.
		\end{equation*}
		Note that the sets $B_n$, $n \in \NN$, form a non-decreasing sequence, 
		and also $[0,1)=\bigcup_{n \in \NN} B_n.$ Take such a set $B_K$ for 
		which $\mu(B_K') \leq \epsilon$. Moreover, introduce 
		\[m_{k,x}=\sum_{i=1}^k e_i(x)\quad\text{and}\quad
		y_{k,x}=2^kx-\lfloor 2^kx \rfloor\quad
		\text{for any }k\in\NN,\;x\in S.\]
		Further, let $B_K=B_K^+ \cup 
		B_K^-$, where $B_K^+$ and $B_K^-$ are disjoint sets, determined by
		\begin{equation*}
		x \in B_K^+ \iff \sum_{i=K+1}^\infty e_i(x)=0,
		\end{equation*}
		\begin{equation*}
		x \in B_K^- \iff \sum_{i=1}^K e_i(x) \geq N \;\wedge\; x \not\in B_K^+ .
		\end{equation*}  
		Note that, if $x \in B_K^+ $, then $\PP^K \delta_x =  \delta_0$, while 
		in the case where $x \in B_K^- 
		$ we get \hbox{$\PP^K \delta_x = \left( 1-2^{-m_{K,x}} \right) \delta_0 + 
		2^{-m_{K,x}} \delta_{y_{K,x}}$ and $m_{K,x} \ge N$}. Hence, for $n 
		\geq K$, we obtain
		\settowidth{\twidth}{$\le |$}
		\begin{equation*}
		\begin{split}
		\PP^n \mu (A)
		=&\int\limits_{B_K^+} \PP^n 
		\delta_x(A) \mu(dx)+\int\limits_{B_K^-} \PP^n \delta_x(A) 
		\mu(dx)+\int\limits_{B_K'} 
		\PP^n \delta_x(A) \mu(dx)\\
		=&\delta_0(A)\mu(B_K^+)+\int\limits_{B_K^-} \left( 1-2^{-m_{K,x}} 
		\right) 
		\delta_0(A)  \\ 
		&+ 2^{-m_{K,x}} \PP^{n-K}\delta_{y_{K,x}}(A) \mu(dx)+\int\limits_{B_K'} 
		\PP^n \delta_x(A) \mu(dx),
		\end{split}
		\end{equation*}
		which implies
		\begin{equation*}
		\begin{split}
		\PP^n\mu(A) &\leq \delta_0(A)\mu(B_K^+)+\delta_0(A) 
		\mu(B_K^-)+2^{-N}\mu(B_K^-)+\mu(B_K') \\ &\leq \delta_0(A) 
		\mu(B_K)+2\epsilon\leq \delta_0(A)+2\epsilon,
		\end{split}
		\end{equation*}
		and
		\begin{equation*}
		\begin{split}
		\PP^n\mu(A)&\geq \delta_0(A)\mu(B_K^+)+\int\limits_{B_K^-} \left( 
		1-2^{-m_{K,x}} 
		\right) 
		\delta_0(A) \mu(dx) \\ &\geq \delta_0(A) \mu(B_K^+) 
		+\left(1-2^{-N}\right) 
		\delta_0(A) \mu(B_K^-)\\
		&\geq (1-\epsilon) \delta_0(A) \mu(B_K)\geq (1-\epsilon)^2 
		\delta_0(A) \geq 
		\delta_0(A) - 2 \epsilon.
		\end{split}
		\end{equation*}
		We have shown that $\PP^n \mu (A) \to \delta_0(A)$, as $n \to \infty$, so, 
		by the Portmanteau Theorem, we obtain that $\PP$ is an asymptotically stable operator.
		
		Let us now investigate the e-property of $\PP$ in $\CC_b(S)$. Take $z 
		\in S$ such 
		that $\sum_{i \in \NN} e_i(z)= \infty$, and let \hbox{$f \in \CC_b(S)$}, $\epsilon>0$. Further, choose $N \in \NN$, satisfying 
		$2\|f\| 2^{-N} \leq 
		\epsilon$. Note that there exists $K \in \NN$ such that $e_K(z)=0$ and 
		$m_{K,z} \geq N$. We can define the neighborhood $\widehat{O}_z$ of $z$ 
		consisting of points $x\in S$ such that $e_i(z)=e_i(x)$ for $i \leq K$. Then,  
		for any $x \in \widehat{O}_z$, we get 
		\settowidth{\twidth}{$= |$}
		\newlength{\ttwidth}
		\settowidth{\ttwidth}{$= \sup |$}
		\begin{equation*}
		\begin{split}
		\sup_{n \geq K}|U^nf(x)-U^nf(z)|
		=&\sup_{n \geq 
			0} \left|\bk{f}{(1-2^{-m_{K,x}})\delta_0+2^{-m_{K,x}} \PP^n 
			\delta_{y_{K,x}}} \right. \\
		&- \left. 
		\bk{f}{(1-2^{-m_{K,z}})\delta_0+2^{-m_{K,z}} 
			\PP^n 
			\delta_{y_{K,z}}}\right|\\
		=&\sup_{n \geq 0} \left|\bk{f}{2^{-m_{K,z}} 
			\PP^n\delta_{y_{K,x}}-2^{-m_{K,z}} \PP^n 
			\delta_{y_{K,z}}}\right|\\
		\leq &2 
		\|f\|2^{-N}\leq \epsilon.
		\end{split}
		\end{equation*}
		Using the Feller property of $\PP$, and recalling that $\epsilon>0$ was 
		chosen arbitrarily, we obtain that $\PP$ has the e-property in 
		$\CC_b(S)$ at $z \in S$.
		
		Now, take the point $z \in S$, satisfying $\sum_{i \in \NN} 
		e_i(z)< \infty$, and let \hbox{$f \in \LL_b(S)$}  be such that  $f(0)=0$ and $f(1/2)=1$. Choose $K \in \NN$, 
		for which 
		\linebreak\hbox{$\sum_{i > K} e_i(z)=0$}. Further, fix the sequence of 
		points 
		$x_n=z+2^{-K-n}$, $n 
		\in \NN$, that converges to $z$, as $n \to \infty$. Then we obtain the 
		following conclusion:
		\settowidth{\twidth}{$\geq |$}
		\begin{equation*}
		\begin{split}
		\limsup_{n \to \infty} \sup_{k \in \NN} 
		&|U^kf(x_n)-U^kf(z)|
		\geq\limsup_{n \to 
			\infty} |U^{K+n-1}f(x_n)-U^{K+n-1}f(z)|\\
		&=\limsup_{n \to 
			\infty} 
		\left|\bk{f}{(1-2^{-m_{K,z}})\delta_0+2^{-m_{K,z}}\PP^{n-1} 
			\delta_{2^{-n}} }\right|
		=2^{-m_{K,z}},
		\end{split}
		\end{equation*}
		and, as a consequence, we see that $P$ does not have the e-property in 
		$\CC_b(S)$ at any 
		dyadic rational point.

\subsection*{Example 2}\label{ex2}

	Let $S=[-2,-1] \cup [0,1]$ and consider the Smith-Volterra-Cantor (SVC) set constructed on the interval $[-2,-1]$. For the sake of constructing an appropriate Markov operator, let us first discuss the construction of the SVC set.
	
	We start by deleting an open interval $w_{1,1}$, having length $1/4$, from the middle of $ [-2,-1]$. The union of the remaining intervals forms a closed set, say $C_1$. Suppose now, that in the $n-$th step the sets $C_1,\ldots,C_{n}$ are well-defined and satisfy 
	$C_n \subset C_{n-1} 
	\subset \ldots 
	\subset C_1$. 
The set $C_n$ is a union of $2^n$ disjoint and closed 	intervals 
	$c_{n,k}$, $k \in \{1,\ldots, 2^n\}$. Moreover, we have
	\begin{equation*}
		|c_{n,k}|=\frac{2^n+1}{2 \times 4^n}.
	\end{equation*}
	To obtain $C_{n+1}$, we delete $w_{n+1,k}$, that is, an open interval of 
	the 
	length $4^{-(n+1)}$, from the middle of any $c_{n,k}$. Then,  
	$C_{n+1}$ is a closed set consisting of $2^{n+1}$ disjoint intervals, having length
	\begin{equation*}
		|c_{n+1,k}|=\frac{|c_{n,k}|-4^{-(n+1)}}{2}=\frac{2^{n+1}+1}{2 \times 
			4^{n+1}}.
	\end{equation*}
	The SVC set is defined as $C=\bigcap_{n \in \NN} C_n.$ One can show that 
	$C$ is a closed, nowhere dense set, having Lebesgue measure equal to $1/2$. 
	
	Let us now define a function $T: [-2,-1] \to [0,1]$. Provided that \hbox{$w_{n,k}=(a,b)$} for some $a,b \in \RR$, $|b-a|=4^{-n}$, we may introduce maps \hbox{$T_{n,k}: (a,b) \to [0,1]$}, $n \in \NN$, $1 \leq k \leq 2^{n-1}$, 
	given by
	\begin{equation*}
		T_{n,k}(x)=\frac{4^{2n+1}}{n}(b-x)(x-a) \quad \text{for } x \in (a,b).
	\end{equation*}
	Next, for any $n \in \NN$, let us define $T_n:[-2,-1] \to [0,1]$ as
	\begin{equation*}
		T_n(x)=\sum_{k=1}^{2^{n-1}} T_{n,k}(x) \1_{w_{n,k}}(x) \quad \text{for 
		} x \in [0,1].
	\end{equation*}	
	Functions $T_n$, $n \in \NN$, are continuous. Moreover, they have disjoint supports and $T_n([-2,-1]) \subset [0,1/n]$ for any $n \in \NN$. That means that the series 
	$\sum_{n \in \NN} T_n$ 
	fulfills the Cauchy criterion, so it converges to some 		continuous function, say $T$.
	
	Define $\pi: S \times \BR(S) \to \RR$ by the formula
	\begin{equation*}
	\begin{split}
		\pi(x,A)&=\delta_{T(x)}(A), \quad x \in [-2,-1]\\
		\pi(x,A)&=\delta_{2x} (A), \quad x \in [0,1/2) \\
		\pi(x,A)&=(2x-1)\delta_0 (A) + (2-2x)\delta_1 (A), \quad x \in 
			[1/2,1]
		\end{split}
	\end{equation*}
where $A \in \BR(S)$. As $\pi$ is a transition probability function, we can 
	define $\PP$ 
	to be a Markov operator 
	generated by $\pi$, according to the rule given in 
	\eqref{pi_P_U}. One can show, that $\PP$ enjoys the Feller property, 
	and notice that $\delta_0$ is an invariant measure of $\PP$. Moreover, for 
	an arbitrary $\mu \in \MM_1(S)$, we have $\text{supp} \, \PP \mu  \subset 
	[0,1]$.  Let $\mu \in \MM_1(S)$, $A \in \BR(S)$ and $\epsilon>0$. 
	Note that there exists $N\in\NN$ such that the set 
	\begin{equation*}
		B=\{0\} \cup [2^{-N},1]
	\end{equation*}
	satisfies $\PP \mu (B') \leq \epsilon$. Further, 			observe that if $x 
	\geq 1/2$, then $\PP^2 \delta_x=\delta_0$. Analogously, for $x \in 
	[1/4,1/2)$ we have $\PP^3 \delta_x=\delta_0$, and so on. Then, for $n \geq 
	N+2$ we obtain
	\begin{equation*}
	\begin{split}
		\PP^n \mu (A) &=\int\limits_B \PP^{n-1} \delta_x(A) \PP \mu 
		(dx)+\int\limits_{B'} 
\PP^{n-1} 
		\delta_x(A) \PP \mu (dx)\\ &=\delta_0(A)\PP\mu(B)+\int\limits_{B'} 
		\PP^{n-1} 
		\delta_x(A) 
		\PP \mu (dx),
		\end{split}
	\end{equation*}
	which implies 
	\begin{equation*}
		\delta_0(A)-\epsilon \leq \delta_0(A)\PP\mu(B)\leq \PP^n \mu(A) \leq 
		\delta_0(A)+\PP \mu(B') \leq 
		\delta_0(A)+\epsilon
	\end{equation*}
	for any $n\geq N+2$. 
	Using the Portmanteau Theorem, we obtain that $\PP$ is asymptotically stable.
	
	Now, let us investigate the e-property of $\PP$ in $\CC_b(S)$. Fix $z > 0$ 
	and $f\in 
	\CC_b(S)$. There 
	exists $M \in \NN$ such that $z > 1/2^M$. Whenever $y \geq 1/2^M$, we have
	\begin{equation*}
		\sup_{n \geq M+1} |U^nf(y)-U^nf(z)|=|f(0)-f(0)|=0.
	\end{equation*}
	The Feller property of $\PP$ then yields that $\PP$ has the e-property in 
	$\CC_b(S)$ at 
	any point $z>0$. If $z \in [-2,-1] \cap C'$, then, by the construction of $T$, 
	there 
	exists a~neighborhood $\widehat{O}_z$ of $z$ such that,  for some $m \in \NN$ and any  
	$y \in \widehat{O}_z$, we have $T(y)>1/2^m$, but this, according to the 
	reasoning presented above, leads us to the conclusion 
	that at every point of $[-2,-1] \cap C'$ operator $\PP$ has the e-property 
	in $\CC_b(S)$. 
	
	We will show that at any point of $ C \cup \{0\}$ (which has measure 
	$1/2$) the operator $\PP$ does 
	not have the e-property in $\CC_b(S)$. Let's start with $z=0$. Take 
	$f(x)=x$. Then,
	\begin{equation*}
		\lim_{k \to \infty} \sup_{n \in \NN} \left|U^nf(1/2^k)-U^nf(0)\right| \geq \lim_{k 
			\rightarrow \infty} U^k f(1/2^k) = f(1)=1.
	\end{equation*}
	Taking $z \in C$, we can choose such a point $x\in C'$, which is arbitrarily close to $z$. 
	Then, for some $n_0 \in \NN$, we have $T(x) \in [1/2^{n_0+1},1/2^{n_0}]$, 
	and 
	therefore we finally obtain
	\begin{equation*}
		\left|U^{n_0+1}f(x)-U^{n_0+1}f(z)\right|=\left|f\left(2^{n_0} T(x)\right)-f(0)\right| \geq 1/2.
	\end{equation*}

\section*{Acknowledgments} 
	
Ryszard Kukulski acknowledges the support from the National Science Centre, 
Poland, under project number 2016/22/E/ST6/00062. The work of Hanna 
Wojew\'odka-\'Sci\k{a}\.zko is supported by the Foundation for Polish Science 
(FNP) under grant number  POIR.04.04.00-00-17C1/18-00.

\bibliographystyle{ieeetr}
\bibliography{markov_ec_rk_hws}
	
\end{document}